\documentclass[11pt,a4paper]{article}

\usepackage{amsmath,amsthm,amssymb,amsfonts,graphicx}
\usepackage[top=3cm, bottom=3cm, left=3cm, right=3cm]{geometry}
\usepackage[shortlabels]{enumitem}
\usepackage{booktabs}

\usepackage{tikz}
\usetikzlibrary{shapes.geometric}
\usetikzlibrary{calc,intersections}
\usetikzlibrary{patterns}

\usepackage{hyperref}

\newcommand{\R}{\mathbb{R}}

\newcommand{\psd}{\ensuremath{\mathcal{PSD}}}
\newcommand{\nn}{\ensuremath{\mathcal{N}}}
\newcommand{\sym}{\ensuremath{\mathcal{S}}}

\newcommand{\cop}{\mathcal{COP}}
\newcommand{\cp}{\mathcal{CP}}
\newcommand{\spn}{\mathcal{SPN}}

\DeclareMathOperator{\trace}{trace}
\DeclareMathOperator{\rank}{rank}

\DeclareMathOperator{\supp}{supp}
\DeclareMathOperator{\diag}{diag}

\DeclareMathOperator{\cpr}{cpr}
\DeclareMathOperator{\tf}{tf}

\newtheorem{theorem}{Theorem}[section]
\newtheorem{lemma}{Lemma}[section]
\newtheorem{corollary}{Corollary}[section]

\newtheorem{proposition}{Proposition}[section]
\newtheorem{observation}{Observation}[section]

\def\x{\mathbf x}

\def\u{\mathbf u}
\def\v{\mathbf v}
\def\d{\mathbf d}
\def\w{\mathbf w}
\def\0{\mathbf 0}
\def\e{\mathbf e}
\def\b{\mathbf b}
\def\1{\bf 1}

\def\cob#1{{\color{blue}#1}}

\title{On the DJL conjecture for order 6\thanks{This work was supported by grant no.\ G-18.-304.2/2011 by the German-Israeli Foundation for Scientific Research and
Development (GIF). } \\
}

\author{
Naomi Shaked-Monderer\thanks{The Max Stern Yezreel Valley College, Yezreel Valley 19300, Israel.
Email: nomi@tx.technion.ac.il}
}

\date{\today}

\begin{document}
\maketitle

\begin{abstract}
\noindent

In 1994 Drew, Johnson and Loewy conjectured that  for $n \ge 4$, the cp-rank of any $n\times n$ completely positive matrices is at most $\lfloor{n^2}/{4}\rfloor$. Recently
this
conjecture has been proved for $n=5$ and disproved for $n\ge 7$, leaving the case $n=6$ open.
We make a step toward proving the conjecture for $n=6$. We show that if $A$ is a $6\times 6$ completely positive matrix that is orthogonal to an exceptional
extremal
copositive matrix,  then the cp-rank of $A$ is at  most $9$.

\medskip



\end{abstract}

\section{Introduction}
A square matrix $A$ is \emph{completely positive} if it has a factorization \begin{equation}\label{eq:cp factorization} A=BB^T, ~~B\ge 0,\end{equation}
where $B$ is not necessarily square. For $A\ne 0$, the minimal
number of columns in such $B$ is the \emph{cp-rank} of $A$, denoted here by $\cpr(A)$. The factorization (\ref{eq:cp factorization}) is a \emph{cp-factorization} of $A$; if the
number of columns of $B$ is   $\cpr(A)$,  (\ref{eq:cp factorization})  is a \emph{minimal cp-factorization}. Finding a tight upper bound on the cp-ranks of $n\times n$
completely
positive matrices is one of the basic problems in the theory of completely positive matrices.

Let $\cp_n$ denote the set of all $n\times n$ completely positive matrices, and let \[p_n=\max_{A\in \cp_n} \cpr(A). \]
For $n\le 4$ it is long known that $p_n=n$ (see, e.g., \cite[Theorem 3.3]{BermanShaked03}).
It was conjectured by Drew, Johnson and Loewy in 1994 that $p_n= \lfloor\frac{n^2}{4}\rfloor$ for every $n\ge 4$
\cite{DrewJohnsonLoewy94}.
The proof for $n=5$  was finally completed only a couple of years ago \cite{LoewTam03, ShakBomJarScha13}.
However, recently this conjecture, the DJL conjecture, was disproved by Bomze, Schachinger and Ullrich, who presented counter examples for any $n\ge 7$, and showed that asymptotically
$p_n$ is of the order $\frac{n^2}{2}$  \cite{BomzeSchachingerUllrich14, BomzeSchachingerUllrich14b}.

A tight upper bound on the cp-rank of a rank $r$, $r\ge 2$, completely positive matrix (of any order)
is known \cite{HannaLaffey83, BarioliBerman03}: $\frac{r(r+1)}{2}-1$, see also \cite[Section 3.2]{BermanShaked03}. This yields the upper bound $\frac{n(n+1)}{2}-1$ on $p_n$, but
this bound is not tight: in \cite{ShakBerBomJarScha14} it was shown that the maximum cp-rank of an $n\times n$ completely positive matrix, $n\ge 5$, is not greater than
$\frac{n(n+1)}{2}-4$. By \cite{BomzeSchachingerUllrich14b}, for $n\ge 15$
\[p_n \ge \frac{n(n+1)}{2}-4-n\left(\sqrt{2n}-\frac{3}{2}\right).\]

Finding an exact tight upper bound on the cp-ranks of $n\times n$ matrices of order $n\ge 6$
is still an open problem, and it is not known whether the DJL bound holds for $n=6$. In \cite{ShakBomJarScha13} it was proved that for every $n$, $p_n$
is attained at a nonsingular matrix on the boundary of $\cp_n$. Thus to prove the DJL conjecture for $n=6$ it suffices to consider the cp-ranks of (nonsingular) matrices on the
boundary of the cone $\cp_6$. In this paper it is shown that for every matrix $A$ on some part of the boundary of $\cp_6$ where $p_6$ may be attained, $\cpr(A)\le 9=6^2/4$. This
part of the boundary includes all the \emph{positive} nonsingular matrices on the boundary of $\cp_6$.

To state the result explicitly, we note that $\cp_n$ is a closed convex cone in the space $\sym_n$ of real
$n\times n$ symmetric matrices, which is a Euclidean space with the inner product
\[\langle A,B\rangle=\trace(AB).\]
The \emph{dual} of a cone $\mathcal{K}\subseteq \sym_n$  is defined by
\[\mathcal{K}^*=\{A\in \sym_n | \langle A,B\rangle\ge 0 \mbox{ for every } B\in \mathcal{K}\},\]
and if $\mathcal{K}$ is closed and convex, its boundary consists of matrices that are orthogonal to extremal matrices in the convex cone $\mathcal{K}^*$.
The dual of the cone $\cp_n$ is the closed convex cone $\cop_n$ of copositive matrices.
A matrix $A\in \sym_n$ is \emph{copositive} if $\x^TA\x\ge 0$ for every nonnegative vector $\x\in \R^n$.
Each positive semidefinite matrix is copositive,
and so is each symmetric nonnegative matrix. A matrix which is a sum of a positive semidefinite
matrix and a nonnegative matrix, called an \emph{SPN matrix}, is also copositive.
A matrix which is copositive but not SPN is called \emph{exceptional}. For $n\ge 5$ there exist exceptional
matrices in $\cop_n$. In $\cop_n$ there are positive semidefinite extremal matrices, nonnegative extremal matrices, and for $n\ge 5$ also exceptional extremal matrices.
Accordingly, for $n\ge 5$ the boundary of $\cp_n$ consists of three (not mutually disjoint) parts: singular matrices, matrices with some zero entries, and matrices orthogonal to
exceptional
extremal matrices. Since, as mentioned above, $p_n$ is attained at a nonsingular matrix on the boundary
of $\cp_n$, it is attained either at a matrix with some zero entries, or at a matrix orthogonal to an exceptional extremal matrix in $\cop_n$.
The  main result of this paper is:

\begin{theorem}\label{thm:main}
Let $A\in \cp_6$ be orthogonal to an exceptional extremal matrix $M\in \cop_6$.  Then $\cpr (A)\le 9$.
\end{theorem}

To prove the theorem we rely on some known results. In particular we need results on minimal cp-factorizations and the cp-rank, some of them in terms
of the zero-nonzero pattern of the completely positive matrix, described by a graph. We also need results on extremal copositive matrices. In Section 2 the needed known results and the relevant concepts are recalled. Theorem \ref{thm:main} is proved
in Section 3.

\section{Preliminaries}\label{prelim}
\subsection{Notation and terminology}
We denote by  $|\alpha|$ the number of elements in a set $\alpha$. The cone of nonnegative vectors in $\R^n$ is denoted by  $\R^n_+$.
Vectors are denoted by bold lower case letters, and the $i$th entry of a vector $\x$ is denoted by $x_i$.  A vector of all ones is denoted by $\1$ and a zero vector by $\0$. The
standard basis vectors in $\R^n$
are $\e_1,\dots, \e_n$.
For a vector $\x\in \R^n$, the \emph{support} of $\x$ is
$\supp \x=\{1\le i\le n\, |\, x_i \ne 0\}$.
The  space of all $m\times n$ real matrices is denoted by $\R^{m\times n}$, and the cone of nonnegative matrices in
this space is denoted by $\R^{m\times n}_+$.
For $M\in \R^{m\times m}$ and $N\in \R^{n\times n}$, $M\oplus N$ is the direct sum of $M$ and $N$.
The vector of diagonal elements of a matrix $A\in \R^{n\times n}$ is denoted by $\diag(A)$.
The matrix $E_{ij}\in \sym_n$ has all entries zero except for the $ij$ and $ji$ entries, which are equal to $1$.
The all ones matrix in $\sym_n$ is denoted by $J_n$ ($J$, when the order is obvious). A \emph{$\pm 1$ matrix} is a matrix all of whose entries are either $1$ or $-1$.
For $A\in \R^{n\times n}$ and $\alpha \subseteq \{1, \dots, n\}$, $A[\alpha]$ denotes the principal submatrix of $A$ on rows and columns $\alpha$, and $A(\alpha)$ the submatrix
induced on the rows and columns other than $\alpha$. We  abbreviate $A[\{i_1,\dots, i_k\}]$  as $A[i_1,\dots, i_k]$, and $A(\{i_1,\dots, i_k\})$  as $A(i_1,\dots, i_k)$. For a
vector $\x\in \R^n$ and $\alpha \subseteq \{1, \dots, n\}$, $\x[\alpha]$ is the vector in $\R^{|\alpha|}$ consisting of
the entries of $\x$ indexed by $\alpha$.
If $A\in \sym_n$ and $B$ is obtained from $A$ by permutation similarity and/or diagonal congruence by a positive diagonal matrix,
we say that $B$ is \emph{in the orbit of $A$}.

Several types of graphs associated with matrices will be used. All graphs in this paper are undirected and simple (no multiple
edges or loops). For graph terminology and notations see \cite{Diestel2010}. We mention here only a few:  The vertex set of a graph $G$ is referred to as $V(G)$, and its edge
set as $E(G)$.
 For a vertex $v\in V(G)$, $d(v)$ denotes the \emph{degree} of $v$, i.e., the number of edges at $v$;
$G-v$ denotes the subgraph of $G$ induced on $V(G)\setminus \{v\}$. For $u,v\in V(G)$, the \emph{distance} between $u$ and $v$ in $G$ is $d_G(u,v)$. The \emph{size} of a graph
$G$ is the number of edges in $G$,  $|E(G)|$. We denote by $\tf(G)$ the size of the
largest triangle free subgraph of $G$. By a theorem of Mantel, the maximum number of edges in a triangle free graph with $n$ vertices is
$\left\lfloor\frac{n^2}{4}\right\rfloor$,
and it is attained by the complete bipartite graph whose independent bipartition sets are as balanced as possible. The complete bipartite graph with independent bipartition sets
of size $m$ and $k$ is denoted by $K_{m,k}$, and $K_{m,1}$ is
a \emph{star}.
For $A\in \sym_n$, the \emph{graph of $A$} is denoted by $G(A)$. It is the graph whose vertex set is $\{1, \dots, n\}$, with $ij$ an edge if and only if $a_{ij}\ne 0$.

\subsection{Minimal cp-factorizations and the cp-rank}
We often use the fact that when $B=(\b_1|\dots |\b_p)$, (\ref{eq:cp factorization}) is equivalent to
\begin{equation}A=\sum_{i=1}^p \b_i\b_i^T, \quad \b_i\in \R^n_+\, .\label{eq:cp decomposition}\end{equation}
The sum (\ref{eq:cp decomposition}) is called a \emph{cp-decomposition} of $A$ (a \emph{minimal cp-decomposition} if $\cpr (A)=p$).
Given a cp-decomposition of $A\in \cp_n$, we may sometimes replace some of the vectors in the decomposition, without changing
the total number of summands, using the following result, which is Observation 1 in \cite{LoewTam03}.

\begin{proposition}\label{pro:pairmove}
Let $\b,\d\in \R^n_+$ such that $\supp \b\subseteq\supp \d$. Then there exist vectors $\tilde{\b}, \tilde{\d}\in \R^n_+$ such that
$\tilde{\b}\tilde{\b}^T+\tilde{\d}\tilde{\d}^T=\b\b^T+\d\d^T$, $\supp \tilde{\d}=\supp \d$, $\supp \tilde{\b}\subseteq \supp \tilde{\d}$, $\supp \d\setminus \supp \b\subseteq
\supp \tilde{\b}$,
and for at least one $i\in \supp \b$, $i\notin \supp\tilde \b$.
\end{proposition}

In particular, if we start with a minimal cp-decomposition of $A$, and apply the previous proposition
repeatedly (at each step replacing a pair  of vectors whose equal supports are the  largest in size), we get:

\begin{proposition}\label{pro:difsupp}
Let $A\in \cp_n$. Then it has a minimal cp-decomposition $A=\sum_{i=1}^p \b_i\b_i^T$, $\b_i\in \R^n_+$, where $\supp\b_i$, $i=1, \dots, p$, are $p$
different sets.
\end{proposition}

The next result is Theorem 5.6 in \cite{ShadShakSzyld14}. It implies that any cp-decomposition of a $3\times 3$ positive completely positive matrix $A$ can be replaced
by a cp-decomposition with the same number of summands, where all the summands are rank $1$ \emph{positive} matrices.  To state it,
we recall a definition from   \cite{ShadShakSzyld14}:
A nonnegative matrix $B$ is called \emph{nearly positive} if there exists a sequence $Q(\ell)$ of orthogonal matrices
converging to $I$ such that $Q(\ell)B>0$ for every $\ell$.

\begin{proposition} \label{pro:NP}
Let $B\in \R^{m\times 3}_+$. Then $B$ is nearly positive if and only if $B^TB>0$.
\end{proposition}

Next we mention results on the cp-rank involving graphs.
Note that if a matrix $B$ is in the orbit of a symmetric matrix $A\in \sym_n$, then $B$ is completely positive if and only if $A$ is, and
$\cpr(B)=\cpr(A)$. Thus we may symmetrically scale our matrices, and when considering graph theoretic results on the cp-rank,
we may re-label the vertices of the graph as we wish.
For a graph $G$, we define
\[\cpr{(G)}= \max\{ \cpr (A)| A \mbox{ is completely positive and } G(A)=G\}.\]
Basic results on the parameter $\cpr(G)$ were collected in \cite{Shak14}. The next proposition  appears there as  Lemmas 3.2.

\begin{proposition}\label{pro:subgraph cpr}
Let $G'$ be  a subgraph of a graph $G$. Then
$\cpr(G')\le \cpr(G)$.
\end{proposition}

In particular, Proposition \ref{pro:subgraph cpr} implies that $\cpr(G)\le p_n$ for every graph $G$ on $n$ vertices.
Another relevant result is the following, which is Lemma 3.3 in \cite{BermanShaked03}.

\begin{proposition}\label{pro:d(v)<3}
Let a graph $G$ have a non-isolated vertex $v$ with $d(v)\le 2$. Then
\[\cpr(G)\le d(v)+\cpr(G-v). \]
\end{proposition}

Several known bounds on the cp-rank of a matrix were given in terms of the its graph. One such example
is the next proposition, originally Theorem 6 in \cite{DrewJohnsonLoewy94}.

\begin{proposition}\label{pro:tf}
Let $G$ be a triangle free graph on $n$ vertices. If $A$ is a completely positive matrix with $G(A)=G$, then
\[\cpr(A)=\max(n, |E(G)|).\]
In particular, $\cpr(A)\le \frac{n^2}{4}$.
\end{proposition}

A matrix $A\in \R^{n\times n}$
is \emph{diagonally dominant} if $|a_{ii}|\ge \sum_{j\ne i}|a_{ij}|$ for every $i=1, \dots, n$.
For the proof of the previous result
it was shown in \cite{DrewJohnsonLoewy94} that every matrix whose graph is triangle free is in the orbit of a
diagonally dominant matrix.

\begin{proposition}\label{pro:being dd}
Let $A\in \sym_n$ be nonnegative. Then the following are equivalent:
\begin{enumerate}
\item[{\rm(a)}] $A$ is in the orbit of a diagonally dominant matrix.
\item[{\rm(b)}] $A=\sum_{i=1}^k\b_i\b_i^T$, where $\b_i\in \R^n_+$ and $|\supp\b_i|\le 2$ for every $i$.
\end{enumerate}
\end{proposition}

The following  generalization of Proposition \ref{pro:tf} to matrices with any graph was proved in \cite{BermanShaked98}.

\begin{proposition}\label{pro:cpr for dd}
Let a nonnegative $A\in \sym_n$ be in the orbit of a diagonally dominant and nonnegative matrix. Then $\cpr(A)\le \frac{n^2}{4}$.
\end{proposition}

In \cite{Shak14} it is shown that $\cpr(G)\ge \tf(G)$ for every connected graph $G$, and some cases where equality holds are discussed.
An \emph{outerplanar graph} is a graph that can be drawn in the plane so that no two edges cross, and all the vertices lie on the boundary of the outer face.
For such graphs the following was proved \cite[Theorem 5.7]{Shak14}.

\begin{proposition}\label{pro:outerplanar}
Every connected outerplanar graph $G$ on $n$ vertices with $\tf(G)\ge n$ satisfies $\cpr(G)=\tf(G)$.
\end{proposition}

A \emph{wheel} is a graph which consists of a cycle and one additional vertex adjacent to all vertices of the cycle. The wheel  on $n$ vertices is denoted by $W_n$. It is not
outerplanar, but it too satisfies $\cpr(W_n)=\tf(W_n)$, by \cite[Theorem 5.9]{Shak14}.

\begin{proposition}\label{pro:cprWn}
For $n\ge 4$,
\[\cpr(W_n)=\tf(W_n)=\left\{\begin{array}{ll}
                              \displaystyle{\frac{3n-3}{2}}\quad & n ~\rm{ is  ~ odd}\\
                              &\\
                              \displaystyle{\frac{3n-4}{2}} \quad  & n  ~\rm{ is  ~ even}
                            \end{array}
 \right.\]
\end{proposition}

\subsection{Copositive matrices and their zeros}\quad Let $\spn_n$ denote the set of $n\times n$ SPN matrices.
The set $\spn_n$ is  a closed convex cone with a nonempty interior in $\sym_n$, and $\spn_n\subseteq \cop_n$. In \cite{Diananda1962} it was shown that for $n\le 4$ this
inclusion
is an equality.
For $n\ge 5$ the inclusion is strict. The first example of an exceptional copositive matrix was given by A. Horn~\cite{Diananda1962}; it is
called the \emph{Horn matrix}:
\begin{equation}H=\left(\begin{array}{rrrrr}
1&-1&1&1&-1\\
-1&1&-1&1&1\\
1&-1&1&-1&1\\
1&1&-1&1&-1\\
-1&1&1&-1&1
\end{array}
\right).\label{eq:Horn}\end{equation}

If a matrix $B$ is in the orbit of $A\in \sym_n$, then $B$ is SPN if and only if  $A$ is, and it is copositive if and only if $A$ is.
Thus $B$ is an exceptional copositive matrix if and only if $A$ is. Also,
$B$ is an extremal copositive matrix if and only if $A$ is.
If the diagonal of a matrix $A$ is positive, then there is a matrix $B$ in the orbit of $A$ with diagonal $\1$ ($B=DAD$, where $D$ is the diagonal matrix with $\diag
D=\left(1/\sqrt{a_{11}},\dots, 1/\sqrt{a_{nn}}\right)$). We therefore often assume that $\diag(A)=\1$, as in the next several propositions:


\begin{proposition}\label{pro:cop diag 1}
Let  $A\in \cop_n$ be an extremal copositive matrix\cob{, with at least one positive diagonal entry}.
\begin{itemize}
\item[{\rm(a)}] If $a_{ii}=0$, then $a_{ij}= 0$ for every $i\ne j$, and $A(i)\in \cop_{n-1}$ is extremal.
\item[{\rm(b)}] If $\diag (A) =\1$, then $a_{ij}\in [-1,1]$ for every $i\ne j$.
\end{itemize}
\end{proposition}

For $A\in \sym_n$ let $G_{-1}(A)$ be the graph whose
vertex set is $\{1,\dots, n\}$ and  $ij$ is an edge of the graph if and only if $a_{ij}=-1$.
The next two propositions are Lemma  3.4 and Lemma 3.5 in \cite{ShakBerDurRaj14}. They characterize positive semidefinite matrices and SPN matrices with
diagonal $\1$ and a connected  $G_{-1}(A)$.

\begin{proposition}\label{pro:psdG-1connected}
Let $A\in \psd_n $ have $\diag A=\1$. If  $G_{-1}(A)$ is connected, then
 $\rank A=1$. In particular, $A$ is a $\pm 1$~matrix and $G_{-1}(A)$ is a complete bipartite graph.
\end{proposition}

\begin{proposition}\label{pro:spnG-1connected}
Let $A\in \sym_n$ have $\diag A={\1}$ and $a_{ij}\ge -1$ for every $i, j$, and let $G_{-1}(A)$ be  connected. Then $A\in \spn_n$ if and only if the following two conditions are
satisfied: $G_{-1}(A)$ is bipartite
and $a_{ij}\ge 1$ whenever $d_{G_{-1}(A)}(i,j)$ is even.
\end{proposition}

A \emph{zero} of a matrix $A\in \cop_n$ is a nonzero vector $\u\in \R^n_+$ such that $\u^TA\u=0$. We will use the following additional terms defined in \cite{Hildebrand14}:  The
zero $\u$ is \emph{minimal} if
no other zero of $A$ has support which is strictly contained in $\supp\u$. A set $\sigma\subseteq \{1, \dots, n\}$ is called
a \emph{zero support} of $A$ if it is the support of a zero of $A$; it is a \emph{minimal support} of $A$ if it is the support
of a minimal zero of $A$. The set of all zeros of $A$ is denoted by $\mathcal{V}^A$, i.e.,
\[\mathcal{V}^A=\{\u\in \R^n_+\setminus\{\0\}|  \u^TA\u=0\}.\]
Zeros and minimal zeros are useful in studying extremal copositive matrices.  In the next four propositions we recall some basic facts. These are Lemma  2.4, Lemma 2.6, Lemma 4.12 and Corollary 4.10 in \cite{DickinsonDuerGijbenHildebrand2013}.

\begin{proposition}\label{pro:A[zs]} Let $A\in \cop_n$, $\u\in \mathcal{V}^A$  and $\sigma=\supp\u$. Then the principal
submatrix $A[\sigma]$ is positive semidefinite, and $\u[\sigma]$ is in the nullspace of $A[\sigma]$.
\end{proposition}

A matrix $A\in \cop_n$ is  called \emph{$E_{ij}$-irreducible} if if $A-\delta E_{ij}\notin \cop_n$  for every $\delta>0$. $A$ is \emph{$\nn$-irreducible} if $A$ is
$E_{ij}$-irreducible
for every  $1\le i, j\le n$, and it
is \emph{$\tilde\nn$-irreducible} if $A$ is $E_{ij}$-irreducible for every  $1\le i\ne j\le n$.
Clearly, any  exceptional extremal copositive matrix  is $\nn$-irreducible (and $\tilde\nn$-irreducible),
but not vice versa.

\begin{proposition}\label{pro:Nirr}
A matrix $A\in \cop_n$ is $E_{ij}$-irreducible if and only if there exists a zero $\u\in \mathcal{V}^A$
such that $(A\u)_i=(A\u)_j=0$ and $u_i+u_j>0$.
\end{proposition}

\begin{proposition}\label{pro:n-1psd of irr}
Let $A\in \cop_n$ be $\tilde\nn$-irreducible. If for some $\sigma\subseteq \{1, \dots, n\}$ with $|\sigma|=n-1$ the submatrix $A[\sigma]$ is positive
semidefinite, then $A$ is positive semidefinite.
\end{proposition}

Propositions \ref{pro:A[zs]} and \ref{pro:n-1psd of irr} imply the next proposition.

\begin{proposition}\label{pro:maxsizesupp}
Let $A\in \cop_n$,  be $\tilde\nn$-irreducible. If some $\u\in \mathcal{V}^A$ has $| \supp(\u)| \ge  n - 1$, then
$A$ is positive semidefinite.
\end{proposition}

Let $A\in \cop_n$ be an exceptional $\nn$-irreducible matrix with $\diag(A)=\1$. It is easy to see (e.g., by Proposition \ref{pro:A[zs]}) that
 a zero of $A$ cannot have support of size $1$, the minimal supports of $A$ are of size at least $2$, and if a minimal support $\sigma$ has two elements, then its two positive
 entries
are equal. Zeros and zero supports were studied in \cite{Hildebrand14}, and the next proposition sums up  Lemma 3.5 and Corollary 3.6 there.

\begin{proposition}\label{pro:zeros}
Let $A\in \cop_n$.
\begin{itemize}
\item[{\rm(a)}] To every minimal support $\sigma$ of $A$ corresponds a unique, up to scalar multiplication, zero of $A$.
\item[{\rm(b)}] Every zero of $A$ is a nonnegative combination of minimal zeros of $A$. Thus every zero support is the
union of minimal supports.
\end{itemize}
\end{proposition}

We also need the following, which is Corollary 3.12 in \cite{Hildebrand14}.

\begin{proposition}\label{pro:union of zero supports}
Let $A$ be a copositive matrix and $\u, \v$ minimal zeros of $A$ such that $|\supp\v \setminus \supp \u|=1$.  Then every zero of $A$ with support
contained in $\supp\u\cup \supp\v$ can be represented as a nonnegative combination of $\u$ and $\v$. In
particular, up to multiplication by a positive scalar, the only minimal zeros with support contained in $\supp \u \cup \supp \v$ are $\u$ and $\v$.
\end{proposition}

The exceptional extremal matrices in $\cop_5$ were completely characterized in \cite{Hildebrand12}. They consist
of the matrices in the orbit of the Horn matrix (\ref{eq:Horn}), and matrices, now called \emph{Hildebrand matrices}. The
 Horn matrix has exactly five minimal suports:
$\{1,2\}$, $\{2,3\}$, $\{3,4\}$, $\{4,5\}$ and $\{1,5\}$.
Its minimal zeros are $\w_i=\e_i+\e_{i\widehat{+}1}\in \R^5$ and its zeros are the vectors of the form $s\w_i+t\w_{i\widehat{+}1}$, $s,t>0$.
 where $\widehat{+}$ denotes summation modulo $5$.
Every Hildebrand matrix has exactly five  zeros, up to multiplication by scalar, all of them minimal, and each with support of size $3$. The minimal supports
are, up to permutations, $\{1,2,3\}$ , $\{2,3,4\}$ , $\{3,4,5\}$ ,  $\{1,4,5\}$ and $\{1,2,5\}$.
(Note that if $B$ is in the orbit of $A\in \cop_n$, then the minimal/zero supports of $B$ are obtained from the minimal/zero supports of $A$ by permutation.)

In \cite{Hildebrand14} all the potential minimal support sets of extremal matrices in $\cop_6$ were found. These are, up to permutation, the
sets in Table 1.

\begin{table}[t]
   \centering
 \begin{tabular}{c}
 Table 1:\\
  Potential minimal support sets of exceptional extremal  $M\in\cop_6$ with $\diag(M)> 0$\\
 ~
\end{tabular}\\
\resizebox{16cm}{!}{
\begin{tabular}{|c|l|c|l|}
\hline
No.& potential minimal supports set& No. & potential minimal supports set\\ \hline
1& \{1,2\},\{1,3\},\{1,4\},\{2,5\},\{3,6\},\{5,6\}& 23 &\{1,2,3\},\{1,2,4\},\{1,2,5\},\{1,3,6\},\{2,4,6\},\{3,4,5,6\}\\
2& \{1,2\},\{1,3\},\{1,4\},\{2,5\},\{3,6\},\{4,5,6\}& 24 & \{1,2,3\},\{1,2,4\},\{1,2,5\},\{1,3,6\},\{3,4,6\},\{3,5,6\}\\
3& \{1,2\},\{1,3\},\{1,4\},\{2,5\},\{3,5,6\},\{4,5,6\}& 25 &\{1,2,3\},\{1,2,4\},\{1,2,5\},\{1,3,6\},\{3,4,6\},\{4,5,6\}\\
4& \{1,2\},\{1,3\},\{1,4\},\{2,5,6\},\{3,5,6\},\{4,5,6\}& 26 & \{1,2,3\},\{1,2,4\},\{1,3,5\},\{1,4,5\},\{2,3,6\},\{2,4,6\}\\
5& \{1,2\},\{1,3\},\{2,4\},\{3,4,5\},\{1,5,6\},\{4,5,6\}& 27 &\{1,2,3\},\{1,2,4\},\{1,3,5\},\{1,4,5\},\{2,3,6\},\{3,4,6\}\\
6& \{1,2\},\{1,3\},\{1,4,5\},\{2,4,6\},\{3,4,6\},\{4,5,6\}& 28 & \{1,2,3\},\{1,2,4\},\{1,3,5\},\{2,4,5\},\{3,4,5\},\{2,3,6\}\\
7& \{1,2\},\{1,3\},\{2,4,5\},\{3,4,5\},\{2,4,6\},\{3,4,6\}& 29 &\{1,2,3\},\{1,2,4\},\{1,3,5\},\{2,4,5\},\{2,3,6\},\{2,5,6\}\\
8& \{1,2\},\{1,3\},\{2,4,5\},\{3,4,5\},\{2,4,6\},\{3,5,6\}& 30 &\{1,2,3\},\{1,2,4\},\{1,3,5\},\{2,4,5\},\{3,4,6\},\{3,5,6\}\\
9& \{1,2\},\{3,4\},\{1,3,5\},\{2,4,6\},\{1,5,6\},\{4,5,6\}& 31 &\{1,2,3\},\{1,2,4\},\{1,3,5\},\{2,4,5\},\{1,5,6\},\{2,5,6\}\\
10& \{1,2\},\{1,3,4\},\{1,3,5\},\{2,3,6\},\{3,4,6\},\{3,5,6\}& 32 &\{1,2,3\},\{1,2,4\},\{1,3,5\},\{2,4,5\},\{1,5,6\},\{4,5,6\}\\
11& \{1,2\},\{1,3,4\},\{1,3,5\},\{1,4,6\},\{2,5,6\},\{3,5,6\}& 33 &\{1,2,3\},\{1,2,4\},\{1,3,5\},\{2,4,5\},\{3,5,6\},\{4,5,6\}\\
12& \{1,2\},\{1,3,4\},\{1,3,5\},\{1,4,6\},\{3,5,6\},\{4,5,6\}& 34 &\{1,2,3\},\{1,2,4\},\{1,3,5\},\{2,4,6\},\{3,5,6\},\{4,5,6\}\\
13& \{1,2\},\{1,3,4\},\{1,3,5\},\{2,4,6\},\{3,4,6\},\{2,5,6\}& 35 &\{1,2,3,4\},\{1,2,3,5\},\{1,2,4,6\},\{1,3,5,6\},\{2,4,5,6\},\{3,4,5,6\}\\
14& \{1,2\},\{1,3,4\},\{1,3,5\},\{2,4,6\},\{3,4,6\},\{3,5,6\}& 36 &\{1,2\},\{1,3\},\{1,4\},\{2,5\},\{4,5\},\{3,6\},\{5,6\}\\
15& \{1,2\},\{1,3,4\},\{1,3,5\},\{2,4,6\},\{3,4,6\},\{4,5,6\}& 37 &\{1,2\},\{1,3,4\},\{1,3,5\},\{1,4,6\},\{2,5,6\},\{3,5,6\},\{4,5,6\}\\
16& \{1,2\},\{1,3,4\},\{1,3,5\},\{2,4,6\},\{3,5,6\},\{4,5,6\}& 38 &\{1,2\},\{1,3,4\},\{1,3,5\},\{2,4,6\},\{3,4,6\},\{2,5,6\},\{3,5,6\}\\
17& \{1,2\},\{1,3,4\},\{2,3,5\},\{3,4,5\},\{2,4,6\},\{3,4,6\}& 39 &\{1,2,3\},\{1,2,4\},\{1,2,5\},\{1,3,6\},\{1,4,6\},\{2,5,6\},\{3,5,6\}\\
18& \{1,2,3\},\{1,2,4\},\{1,2,5\},\{1,3,6\},\{1,4,6\},\{1,5,6\}& 40 &\{1,2,3\},\{1,2,4\},\{1,2,5\},\{1,3,6\},\{1,4,6\},\{3,5,6\},\{4,5,6\}\\
19& \{1,2,3\},\{1,2,4\},\{1,2,5\},\{1,3,6\},\{1,4,6\},\{2,5,6\}& 41 &\{1,2,3\},\{1,2,4\},\{1,2,5\},\{1,3,6\},\{2,4,6\},\{3,4,6\},\{3,5,6\}\\
20& \{1,2,3\},\{1,2,4\},\{1,2,5\},\{1,3,6\},\{1,4,6\},\{3,5,6\}& 42 &\{1,2,3\},\{1,2,4\},\{1,2,5\},\{1,3,6\},\{2,4,6\},\{3,5,6\},\{4,5,6\}\\
21& \{1,2,3\},\{1,2,4\},\{1,2,5\},\{1,3,6\},\{2,4,6\},\{3,4,6\}& 43 &\{1,2,3\},\{1,2,4\},\{1,2,5\},\{1,3,6\},\{1,4,6\},\{2,5,6\},\{3,5,6\},\{4,5,6\}\\
22& \{1,2,3\},\{1,2,4\},\{1,2,5\},\{1,3,6\},\{2,4,6\},\{3,5,6\}& 44 &\{1,2,3\},\{1,2,4\},\{1,3,5\},\{1,4,5\},\{2,3,6\},\{2,4,6\},\{3,5,6\},\{4,5,6\}\\
\hline
\end{tabular}}
\end{table}

A few of the sets in the table have been confirmed as minimal supports sets of exceptional extremal matrices, and some excluded (some by Hildebrand and some by Dickinson), but
for the majority of these sets it is yet unknown whether they are indeed minimal support sets. Since these additional results have not yet been properly published, we will not use
them, and will show that in any case, \emph{if} there exists an exceptional extreme matrix $M$ with one
of these minimal supports set, then any $A\in \cop_6$ orthogonal to $M$ has $\cpr (A)\le 9$.

\section{Proof of the main result}
Given a matrix $M\in \cop_n$ with some zeros, let $\{\sigma_1, \dots, \sigma_k\}$ be the set of its  minimal supports, and let $\w_1, \dots, \w_k$ be
minimal zeros such that $\supp(\w_i)=\sigma_i$. We set
\begin{equation}W=\left(\w_1|\dots |\w_k\right)\in \R^{n\times k}_+ \label{eq:W}\end{equation}
and refer to $W$ as \emph{the matrix of minimal zeros} of $M$. It is, of course, unique only up to permutation of the columns and multiplication on the right
by a positive diagonal matrix (Proposition \ref{pro:zeros}(a)).

\begin{observation}\label{ob:AorthM}
Let $A\in \cp_n$ be orthogonal to $M\in \cop_n$, and let $W\in \R^{n\times k}_+$ be the matrix of minimal zeros of $M$.  If $A=BB^T$ is a cp-factorization of $A$ with
$B\in \R^{n\times m}_+$, then
 there exists a nonnegative $X\in \R^{k\times m}_+$ such that $B=WX$, and for every such $X$, $\cpr(A)\le \cpr(XX^T)$.
\end{observation}

\begin{proof}
By Proposition \ref{pro:zeros}(b), every column of $B$ is a nonnegative combination of the columns of $W$, hence $B=WX$ for some $X\in \R^{k\times m}_+$.
If $YY^T$ is a minimal cp-factorization of $XX^T$ with $Y\in \R^{k\times p}_+$, then
\[A=BB^T=(WX)(WX)^T=W(XX^T)W^T=W(YY^T)W^T=(WY)(WY)^T,\]
where $WY\in \R^{n\times p}_+$, implying that $\cpr (A)\le p$.
\end{proof}

Using the above observation, we can improve the bound in \cite[Proposition 6.1]{ShakBerBomJarScha14} on the cp-ranks of
matrices orthogonal to a matrix $M$ in the orbit of $H\oplus 0$, where $H$ is either the Horn matrix or a Hildebrand matrix.

\begin{lemma}\label{lem:M is H+0}
Let $M\in \cop_6$ be an exceptional extremal matrix with a zero diagonal entry. If $A\in \cp_6$ is orthogonal to
$M$, then $\cpr (A)\le 7$.
\end{lemma}

\begin{proof}
By Proposition \ref{pro:cop diag 1} an extremal matrix in $\cop_6$ with two zero
diagonal entries is a direct sum of a $4\times 4$ SPN matrix and a $2\times 2$ zero matrix, and is therefore SPN. Since $M$ is exceptional, it has exactly one zero entry on the diagonal.
Since $M$ is extremal, it is in the orbit of a matrix $H\oplus 0$, where $H$ is either the Horn matrix or a Hildebrand matrix. We may assume that $M=H\oplus 0$.
For every zero $\u$ of $M$, $\u[1,2,3,4,5]$ is a zero of $H$.

If $H$ is the Horn matrix, the minimal zeros of $M$ are $\w_i=\e_i+\e_{i\widehat{+}1}\in \R^6$, $i=1, \dots, 5$, where $\widehat{+}$ denotes summation modulo 5, and $\w_6=\e_6$.
Let $W=\left(\w_1|\dots |\w_6\right)$
be the matrix of minimal zeros of $M$.
By Observation \ref{ob:AorthM},  $A=(WX)(WX)^T$ for some $X\in \R^{6\times k}_+$, and $\cpr(A)\le \cpr(XX^T)$. Since every zero of $M$ is a nonnegative combination of $\w_i,
\w_{i\widehat{+}1}$ and $\w_6$ for some $1\le i\le 5$, $G(XX^T)$ is a subgraph of the wheel $W_6$ and, by Proposition \ref{pro:cprWn}, $\cpr(XX^T)\le \cpr(W_6)=7$.

If $H$ is a Hildebrand matrix, then $M$ has six minimal zeros: five zeros $\w_1,\dots, \w_5$ obtained by appending a zero entry to each (minimal) zero of $H$, and $\w_6=\e_6$.
As
above,  $A=(WX)(WX)^T$, where $W$ is the matrix of minimal zeros of $M$ and $X\in \R^{6\times k}_+$, and $\cpr(A)\le \cpr(XX^T)$. In this case, every zero of $M$ is a
nonnegative
combination of $\w_i$ and $\w_6$ for some $1\le i\le 5$, so $G(XX^T)$ is a subgraph of the star on $6$ vertices. A star is a tree  and thus,  by Proposition
\ref{pro:tf}, its cp-rank is equal to the number of its vertices. Thus $\cpr(XX^T)\le 6$.
\end{proof}

To find good bounds on the cp-rank for matrices orthogonal to an exceptional extremal matrix $M\in\cop_6$ with
positive diagonal we need also some lemmas about the zero supports of such $M$. We may assume that $\diag(M)=\1$.
Note that in this case each zero support has at least two elements, and thus zero supports of size $2$ are necessarily
minimal. The next lemma states that the union of two non-disjoint size $2$ zero supports of $M$ is also a zero support of $M$.

\begin{lemma}\label{lem:z union}
Let $M\in \cop_n$ be an extremal copositive matrix with $\diag(M)=\1$. If $\{i,j\}$ and $\{j,k\}$ are minimal supports of $M$, then
$\{i,j,k\}$ is a zero support of $M$, and $\{i,k\}$ is not a zero support of $M$.
\end{lemma}

\begin{proof}
W.l.o.g. assume that $i=1$, $j=2$, $k=3$, and let $\sigma=\{1,2,3\}$. Then
\[M[\sigma]=\left( \begin{array}{rrr}
                    1 & -1 & a \\
                    -1 & 1 &-1 \\
                    a & -1 & 1
                  \end{array}
\right).\]
Since $M[\sigma]\in \spn_3$, necessarily $a\ge 1$  by Proposition \ref{pro:spnG-1connected}, and since $M$ is extremal, $a=1$  by Proposition \ref{pro:cop diag 1}(b). It is then
easy to see that there are zeros of $M$ with support $\sigma$ (e.g., $\u=\e_1+2\e_2+\e_3$), while $\{1,3\}$ is not a zero support.
\end{proof}

The next lemma gives a sufficient condition for a union of three zero supports  to be a zero support.

\begin{lemma}\label{lem:zero triangle}
Let $M\in \cop_n$, and let $\sigma_1,\sigma_2, \sigma_3$ be three minimal supports of $M$, such that
$\sigma_i\cup \sigma_j$ is a zero support for every $1\le i\ne j\le 3$. Then $\sigma_1\cup \sigma_2\cup \sigma_3$ is
a zero support of $M$.
\end{lemma}

\begin{proof}
\cob{A corrected proof due to Peter~J.~C.~Dickinson: The result holds  for any number of zero supports $\sigma_i$, $i=1, \dots, m$, such
that $\sigma_i\cup \sigma_j$ is a zero support for every $i\ne j$:
For each $1\le i\le m$ let $\w_i$  be a  zero  of $M$ with support  $\sigma_i$.  
For every $i\ne j$ the principal submatrix $M[\sigma_i\cup \sigma_j]$ is  positive semidefinite, as $\sigma_i\cup \sigma_j$ is a zero support. 
Since  $\supp \w_i \subseteq \sigma_i\cup \sigma_j$, 
$(M\w_i)_l=0$ for every $l\in \sigma_i\cup \sigma_j$, combined with  $\supp \w_j\subseteq \sigma_i\cup \sigma_j$, 
this implies $\w_j^TM\w_i=0$.}

\cob{Let $\w=\sum_{i=1}^m\w_i$. Then $\supp\w=\cup_{i=1}^m\sigma_i$, and
\[\w^TM\w=\sum_{i,j=1}^m\w_j^TM\w_i=0, \]
i.e., $\w$ is a zero of $M$.}
\end{proof} 

Combining the last two lemmas, we get the following corollary:

\begin{corollary}\label{cor:3 size 2 supports}
Let $M\in \cop_n$, and let $\sigma_1,\sigma_2, \sigma_3$ be three different minimal supports of $M$ of size $2$, such that $\sigma_i\cap\sigma_j\ne \emptyset$ for
every $i\ne j$. Then $\sigma_1\cup \sigma_2\cup \sigma_3$ is a zero support of $M$, of size $4$.
\end{corollary}

If $M$ is an exceptional extremal $M\in \cop_6$ whose diagonal is positive, each of its zero supports has at most $4$ elements by Proposition \ref{pro:maxsizesupp}.

\begin{lemma}\label{lem:size 3 supp union}
Let $M\in \cop_6$, and let $\sigma$ be a zero support of $M$ of size $4$. If $\sigma$ contains a minimal support of size $3$, then $\sigma$
contains exactly two minimal supports, and is equal to their union.
\end{lemma}

\begin{proof}
W.l.o.g. assume that $\sigma=\{1,2,3,4\}$ and that $\sigma_1=\{1,2,3\}$ is a minimal support of size $3$ contained in $\sigma$. Let $\u$ be a zero of $M$ with $\supp\u=\sigma$.
Then $\u$ is a nonnegative combination of minimal zeros, and the union of the corresponding minimal supports is $\sigma$. Thus there is at least one
minimal support $\sigma_2\subseteq \sigma$ such that $4\in \sigma_2$. But then $\sigma=\sigma_1\cup \sigma_2$. The result now follows from Proposition \ref{pro:union of zero
supports}.
\end{proof}

\begin{lemma}\label{lem: support 4 zeros}
Let $M\in \cop_6$ be an exceptional extremal matrix with $\diag(M)=\1$. If a zero support $\sigma$ of $M$ contains $3$ or more different minimal supports,
then $|\sigma|=4$ and $M[\sigma]$ is a $\pm 1$ positive semidefinite matrix of rank $1$. Moreover, there are either $3$ or $4$ minimal supports contained in $\sigma$,
each of them of size $2$, and the union of any two of these minimal supports is also a zero support.
\end{lemma}

\begin{proof}
By Lemma \ref{lem:size 3 supp union}, all the minimal supports contained in $\sigma$ are of size $2$.
 Since $|\sigma|\le 4$, the three minimal supports cannot all be pairwise disjoint.
  Suppose $\sigma_1$ and $\sigma_2$ are  size $2$ minimal supports contained in $\sigma$ such that
  \begin{equation}\sigma_1\cap\sigma_2\ne\emptyset.\label{eq:s1caps2}\end{equation} Then $|\sigma_1\cap\sigma_2|=1$ and $|\sigma_1\cup\sigma_2|=3$.
By Proposition \ref{pro:union of zero supports},
$\sigma_1$ and $\sigma_2$ are the only minimal supports contained in $\sigma_1\cup\sigma_2$. Therefore, a third minimal support, $\sigma_3$, satisfies $|\sigma_1\cup
\sigma_2\cup
\sigma_3|=4$. That is,  $\sigma=\sigma_1\cup \sigma_2\cup \sigma_3$, and
\begin{equation}\sigma_3\cap(\sigma_1\cup \sigma_2)\ne \emptyset .\label{eq:cup si}\end{equation}
 Since $M$ is extremal and  $\diag(M)=\1$, all the entries of $M$ are in the interval $[-1,1]$,  
and each minimal support  of size $2$ contained in $\sigma$ corresponds to a $-1$ off diagonal entry in $M[\sigma]$. Since $\sigma$ is a zero support,
the matrix $M[\sigma]$ is positive semidefinite.
By (\ref{eq:s1caps2})
and (\ref{eq:cup si}), $G_{-1}(M[\sigma])$ is connected.
Proposition \ref{pro:psdG-1connected} then implies that  $M[\sigma]$ is a $\pm 1$ positive semidefinite matrix of rank $1$, and
$G_{-1}(M[\sigma])$ is a complete bipartite graph on $4$ vertices. That is, $G_{-1}(M[\sigma])$ is either $K_{1,3}$ or $K_{2,2}$. The minimal supports contained
 in $\sigma$ correspond to the  edges of $G_{-1}(M[\sigma])$, and therefore there are either three or four of them. If $\sigma_i\cap \sigma_j\ne \emptyset$, then $\sigma_i\cup
 \sigma_j$ is a zero support by Lemma \ref{lem:z union}.
If $\sigma_i\cap \sigma_j=\emptyset$ then $\sigma_i\cup \sigma_j=\sigma$, and is therefore a zero support by the initial assumption.
\end{proof}

For an exceptional extremal $M\in \cop_6$ with positive diagonal we define  $\mathcal{G}_{\mathcal{V}}(M)$ to be the graph whose vertex set is the set of minimal
supports of $M$, $\{\sigma_1, \dots, \sigma_k\}$, in which $\sigma_i\sigma_j$ is an edge if and only if $\sigma_i\cup \sigma_j$ is a zero support of $M$.
By Lemmas \ref{lem:size 3 supp union} and \ref{lem: support 4 zeros} each zero support of $M$ corresponds to a
clique on at most $4$ vertices in $\mathcal{G}_{\mathcal{V}}(M)$, and if a zero support is
represented by a clique on three or four vertices, then the vertices of the clique are minimal supports
of size $2$.

Suppose $A\in \cp_6$ is orthogonal to $M$. Let $B$ and $X$ be as in Observation \ref{ob:AorthM}, $B=(\b_1|\dots|\b_m)$.
For every $1\le i\le m$, the  column $\b_i$ can be represented as a nonnegative combination of $\ell_i$ minimal zeros, $\ell_i\le 4$. Thus we may choose $X$ such that
 support of its $i$-th
column is a clique with $\ell_i$ elements in $\mathcal{G}_{\mathcal{V}}(M)$. In particular,
\begin{equation}G(XX^T)\subseteq \mathcal{G}_{\mathcal{V}}(M)\label{eq:GXX^T contained}.\end{equation}
By Observation \ref{ob:AorthM} and Proposition \ref{pro:subgraph cpr}, (\ref{eq:GXX^T contained}) implies
\begin{equation}\cpr(A)\le  \cpr(\mathcal{G}_{\mathcal{V}}(M)).\label{eq:cprAcprGVM}\end{equation}

In some cases the bound in (\ref{eq:cprAcprGVM}) can be improved.

\begin{lemma}\label{lem:XXT tf}
Let $M\in \cop_6$ be an exceptional extremal matrix with $k$ minimal zeros, and let $A\in \cp_6$ be orthogonal to $M$.
If each zero support of $M$ is a union of at most two minimal supports, then
\[\cpr(A)\le \max(k,\tf(\mathcal{G}_{\mathcal{V}}(M))).\]
\end{lemma}

\begin{proof}
Let $B$ and $X$ be as in Observation \ref{ob:AorthM}.  By Lemma \ref{lem:zero triangle} and the assumptions
on $M$, $\mathcal{G}_{\mathcal{V}}(M)$ is a triangle free graph, and so is its subgraph $G(XX^T)$. By Proposition \ref{pro:tf}
and (\ref{eq:GXX^T contained})
\[\cpr(XX^T)\le \max(k,|E(G(XX^T))|)\le \max(k,\tf(\mathcal{G}_{\mathcal{V}}(M))).\]
 The result follows from Observation \ref{ob:AorthM}.
\end{proof}

For most potential minimal zero sets in Table 1 we do not have enough information
on the graph $\mathcal{G}_{\mathcal{V}}(M)$. We therefore
define for each $M$ the graph $\mathcal{G}(M)$ whose vertices are the minimal zero supports $\sigma_1, \dots, \sigma_k$ of $M$,
and $\sigma_i\sigma_j$ is an edge if and only if  $|\sigma_i\cup \sigma_j|\le 4$. Then
\begin{equation}\mathcal{G}_{\mathcal{V}}(M)\subseteq \mathcal{G}(M)\label{eq:G_V(M) subgraph G(M)},\end{equation}
and therefore
\begin{equation}\cpr(\mathcal{G}_{\mathcal{V}}(M))\le \cpr(\mathcal{G}(M))  \   \text{ and } \    \tf(\mathcal{G}_{\mathcal{V}}(M))\le
\tf(\mathcal{G}(M)).\label{eq:G(M)}\end{equation}

We can now prove Theorem \ref{thm:main}.

\noindent\emph{Proof of Theorem \ref{thm:main}.}
Let $M$ and $A$ be as in the statement of the theorem. If $M$ has a zero diagonal entry, then by Lemma \ref{lem:M is H+0}  $\cpr(A)\le 7$. So suppose $M$ has all diagonal
entries
positive. We may assume that $\diag(M)=\bf 1$. The set of minimal supports of $M$ is one of the sets on Table 1.
We will show that $\cpr(A)\le 9$ for each of these potential minimal supports sets.

For a large number of these cases the same short proof applies:\\
\noindent{\bf Sets no. 6-35. } Each of these (potential) minimal support sets has $6$ elements, at most two of them are supports of size $2$. Let $X\in \R^{6\times m}$ be as in
Observation \ref{ob:AorthM}.
By Lemma \ref{lem:size 3 supp union} each zero of $M$ is a nonnegative combination of at most two minimal zeros, so the support of each column of $X$ is of size at most $2$.
Thus $XX^T$ is a $6\times 6$ matrix which is in the orbit of a diagonally dominant matrix (Proposition \ref{pro:being dd}), so $\cpr(XX^T)\le {6^2}/{4}=9$ (Proposition \ref{pro:cpr for dd}). By Observation
\ref{ob:AorthM},
$\cpr(A)\le 9$.

We now consider the remaining sets:\\
\noindent{\bf Sets no. 1-2. }  Let $\sigma=\{2,5\}$. Each of the principal submatrices $M[1,2,4,5]$, $M[1,2,3,5]$, $M[2,3,5,6]$ is SPN  with diagonal equal to $\1$, and  its
$G_{-1}$ graph is  connected and not complete bipartite. By Proposition \ref{pro:psdG-1connected} these principal submatrices are not positive semidefinite. Thus, by Proposition
\ref{pro:zeros}, $\{1,2,4,5\}$, $\{1,2,3,5\}$, $\{2,3,5,6\}$ are not zero supports.
 The possible zero supports containing $\sigma$ are therefore
$\{2,5\}$ itself, $\{1,2,5\}$ and $\{2,5,6\}$ (in case 1) or $\{2,4,5,6\}$ (in case 2). This means that the degree of $\sigma$ as a vertex of $\mathcal{G}_{\mathcal{V}}(M)$ in
both cases is at most 2.
By Proposition \ref{pro:d(v)<3}, and since $\mathcal{G}_{\mathcal{V}}(M)-\sigma$ has 5 vertices,
\[\cpr(\mathcal{G}_{\mathcal{V}}(M))\le 2+\cpr(\mathcal{G}_{\mathcal{V}}(M)-\sigma)\le 2+6.\]
By (\ref{eq:cprAcprGVM}), this implies that $\cpr(A)\le 8$.

\noindent{\bf Sets no. 3-4. } Let $\sigma=\{2,5\}$ (in Set 3) or $\sigma=\{2,5,6\}$ (in Set 4).
In case of Set 3, $M[1,2,4,5]$, $M[1,2,3,5]$ are not positive semidefinite by the same argument used for the  sets 1-2.
Thus the only possible zero supports of $M$ containing $\sigma$, other than $\sigma$ itself, are $\{1,2,5\}$, $\{2,3,5,6\}$ and $\{2,4,5,6\}$ in case 3, and  $\{1,2,5,6\}$, $\{2,3,5,6\}$ and $\{2,4,5,6\}$ in the case of Set 4.

Let $\w_1$ be a minimal zero  of $A$ supported by $\sigma$, and let $\w_2, \w_3$ and $ \w_4$ be minimal zeros supported by $\{1,2\}$, $\{3,5,6\}$ and  $\{4,5,6\}$, respectively.
For a minimal cp-factorization $A=BB^T$ with $B=(\b_1|\dots |\b_m)\in \R^{6\times m}$, let
\[\Omega_1=\{i | \sigma\subseteq \supp \b_i\},  \ \ \Omega_2=\{1, \dots, m\}\setminus \Omega_1.\]
Choose a minimal cp-factorization for which  $|\Omega_1|$ is minimal.
Let $A_1=\sum_{i\in \Omega_1}\b_i\b_i^T$ and $A_2=\sum_{i\in \Omega_2}\b_i\b_i^T$.
Since the cp-factorization is minimal, $\cpr(A_i)=|\Omega_i|$, $i=1,2$, and \[\cpr(A)=\cpr(A_1)+\cpr(A_2).\]
Since each $\b_i$, $i\in \Omega_2$, is a nonnegative combination of  minimal zeros whose support does not contain $\sigma$,
applying Observation \ref{ob:AorthM} and (\ref{eq:cprAcprGVM}) to $A_2$ (and observing that $\mathcal{G}_{\mathcal{V}}(M)-\sigma$ has 5 vertices)  yields that
\[\cpr (A_2)\le \cpr (\mathcal{G}_{\mathcal{V}}(M)-\sigma)\le 6.\]
It thus remains to show that $\cpr(A_1)=|\Omega_1|\le 3$.

If $\Omega_1$ is a singleton, $\cpr(A_1)=1$ and we are done. Otherwise, $\Omega_1$ has at least two elements. In
that case, no $\b_i$, $i\in \Omega_1$, is supported by $\sigma$, otherwise we could apply  Proposition \ref{pro:pairmove} to replace
it and another $\b_j$, $j\in \Omega_1$, by two vectors, one of which with support that does not contain $\sigma$. This would contradict
the  assumption that  $|\Omega_1|$ is minimal. We therefore have for every  $i\in \Omega_1$, $ \sigma\subsetneq \supp \b_i$. Moreover, by the same argument, in any other
cp-decomposition of $A_1$ none of the vectors is supported by $\sigma$.  By Proposition \ref{pro:difsupp} applied
to $A_1$ we
may assume that $\b_i$,
$i\in \Omega _1$, have different supports. Since there are exactly
three zero supports strictly containing $\sigma$,  $\cpr(A_1)\le 3$, and the proof for these two cases is complete.

\noindent{\bf Set  no. 5.} As in the previous cases,  $M[1,2,3,4]$ is not positive semidefinite, and thus   the union of
 $\{1,2\}$, $\{1,3\}$ and $\{2,4\}$ is not a zero support.
Combined with the fact that all
the other minimal supports or $M$ are of size $3$, we get by Lemma \ref{lem:size 3 supp union} that every zero support
of $M$ is the union of at most two minimal zero supports.

By Lemma \ref{lem:XXT tf} and (\ref{eq:G(M)}), \[\cpr (A) \le \max(6,\tf(\mathcal{G}_{\mathcal{V}}(M)))\le  \tf (\mathcal{G}(M))=8.\]
(To compute $\tf(\mathcal{G}(M))$ note that there exist two disjoint triangles in  $\mathcal{G}(M)$ (see Fig. 1), thus at least two of this graph's ten edges need to be removed
to get a triangle free subgraph.  Omit the edges $\{1,2\}\{1,3\}$ and $\{2,4\}\{3,4,5\}$ to get a triangle free subgraph of $\mathcal{G}(M)$ of maximal size.)

\begin{center}
 \begin{tikzpicture}
\draw[semithick] (0,0)--(3,0)--(3,3)--(0,3)--cycle;
\draw[semithick]  (0,3)--(3,0);
\draw[semithick]  (0,0)--(1,2);
\draw[semithick]  (0,0)--(2,1);
\draw[semithick]  (2,1)--(3,3);

\node[left] at (0,3) {\{1,5,6\}};
\node[right] at (3,3) {\{4,5,6\}};
\node[right] at (1,2) {\{1,2\}\quad };
\node[right] at (2,1) {\{2,4\}};
\node[right] at (3,0) {\{3,4,5\}};
\node[left] at (0,0){\{1,3\}};

\draw[fill](0,0) circle[radius=0.1];
\draw[fill](3,0) circle[radius=0.1];
\draw[fill](0,3) circle[radius=0.1];
\draw[fill](3,3) circle[radius=0.1];
\draw[fill](1,2) circle[radius=0.1];
\draw[fill](2,1) circle[radius=0.1];

\node[align=center, below] at  (1.5,-0.7){Fig. 1: $\tf(\mathcal{G}(M))=8$, Case 5};
 \end{tikzpicture}
 \end{center}

\noindent{\bf Set no. 36. } In this case, each minimal zero is of size $2$. Since $\{1,3\}, \{1,2\},\{2,5\},\{5,6\}$ and $ \{3,6\}$
are minimal zeros of $M(4)$ and $\diag(M(4))=\1$, the matrix $M(4)$ is
a permutation of the Horn matrix. Similarly, $M(2)$ is also a permutation of the Horn matrix. This implies that  $M$ is a $\pm 1$-matrix, except possibly the entry $m_{24}$. By
 Lemma \ref{lem:z union} applied to $i=2$, $j=1$ and $k=4$, we also have $m_{24}=1$.
By Proposition \ref{pro:psdG-1connected}, $M[2,3,5,6]$, $M[1,2,3,6]$ and $M[3,4,5,6]$ are not positive semidefinite since their $G_{-1}$ graph is not complete bipartite.
Therefore
the only zero supports of $M$ containing $\{3,6\}$ are $\{3,6\}$ itself, $\{1,3,6\}$ and $\{3,5,6\}$.

The minimal zeros of size $2$ contained in $\{1,2,4,5\}$ imply that $G_{-1}(M[1,2,4,5])$ is a complete bipartite graph, $K_{2,2}$, hence
the submatrix $M[1,2,4,5]$ is a $\pm1$ rank $1$ positive semidefinite matrix, and
$\{1,2,4,5\}$ is a zero support of size $4$ (it is a union of two disjoint minimal supports in two ways: $\{1,2\}\cup \{4,5\}$ and $\{1,4\}\cup\{2,5\}$).
 The nullspace of the positive semidefinite matrix $M[1,2,4,5]$ is spanned by the minimal zeros of this submatrix,
\[\bar\v_1=\left(
         \begin{array}{c}
           1 \\
           1 \\
           0 \\
           0 \\
         \end{array}
       \right) ~ ,~ \bar\v_2=\left(
         \begin{array}{c}
           0 \\
           0 \\
           1 \\
           1 \\
         \end{array}
       \right)  ~ ,~ \bar\v_3=\left(
         \begin{array}{c}
           1 \\
           0 \\
           1 \\
           0 \\
         \end{array}
       \right)  ~ ,~ \bar\v_4=\left(
         \begin{array}{c}
           0 \\
           1 \\
           0 \\
           1 \\
         \end{array}
       \right) .
\]
In fact, it is not hard to see that every zero of $M[1,2,4,5]$ may be represented either as a nonnegative combination of
$\bar\v_1, \bar\v_2, \bar\v_3$, or as a nonnegative combination of $\bar\v_1, \bar\v_2, \bar\v_4$.
Let $\v_1, \v_2, \v_3, \v_4$ the vectors in $\R^6$ obtained by appending zero entries to $\bar\v_1, \bar\v_2, \bar\v_3, \bar\v_4$, so that $\v_i[1,2,4,5]=\bar\v_i$.
Then  every zero of $M$ whose support is contained in $\{1,2,4,5\}$ can be represented as a nonnegative combination  of either $\v_1, \v_2, \v_3$ or  $\v_1, \v_2,
\v_4$.
Let  $W\in \R^{n\times k}_+$ be the matrix of minimal zeros of $M$. Then $A$ has a minimal cp-factorization  $A=BB^T$ with $B=WX$, where $X\in \R^{k\times p}_+$, and $G(XX^T)$
is a subgraph of the  graph $\mathcal{G}$ shown in Fig. 2 (note that $\mathcal{G}_{\mathcal{V}}(M)$ contains also the edge $\sigma_1\sigma_2$, where $\sigma_1=\{1,4\}$,
$\sigma_2=\{2,5\}$, but by the above $X$ can be chosen  so that $G(XX^T)$ does not include that edge). Let $\sigma=\{3,6\}$. Then $\sigma$ is a vertex of degree $2$ in
$\mathcal{G}$, and $\mathcal{G}-\sigma$ is an outerplanar graph
with $\tf(\mathcal{G}-\sigma)=7$ (it has $9$ edges, and  two disjoint triangles). Combining Propositions \ref{pro:d(v)<3} and \ref{pro:outerplanar}  we get that
\[\cpr(A)\le 2+\cpr(\mathcal{G}-\sigma)=2+\tf(\mathcal{G}-\sigma)=9.\]

\begin{center}
 \begin{tikzpicture}

\draw[semithick] (0,0.75)--(0.75,0);
\draw[semithick] (0,0.75)--(0.75,1.5);
\draw[semithick] (0.75,0)--(0.75,1.5);
\draw[semithick] (0.75,0)--(2.25,0);
\draw[semithick] (2.25,0)--(2.25,1.5);
\draw[semithick] (0.75,1.5)--(2.25,1.5);
\draw[semithick] (0.75,1.5)--(2.25,0);
\draw[semithick] (2.25,1.5)--(3,0.75);
\draw[semithick] (2.25,0)--(3,0.75);
\draw[semithick] (1.5,3) to [out=195,in=90] (0,0.75);
\draw[semithick] (3,0.75) to [out=90,in=-15] (1.5,3);

\draw[fill](0.75,0) circle[radius=0.1];
\draw[fill](0,0.75) circle[radius=0.1];
\draw[fill](0.75,1.5) circle[radius=0.1];
\draw[fill](2.25,0) circle[radius=0.1];
\draw[fill](2.25,1.5) circle[radius=0.1];
\draw[fill](3,0.75) circle[radius=0.1];
\draw[fill](1.5,3) circle[radius=0.1];

\node[left] at (0,0.75) {\{1,3\}};
\node[below] at (0.75,0) {\{1,4\}};
\node[above] at (0.75,1.5) {\{1,2\}};
\node[below] at (2.25,0) {\{4,5\}};
\node[above] at (2.25,1.5) {\{2,5\}};
\node[right] at (3,0.75){\{5,6\}};
\node[above] at (1.5,3) {\{3,6\}};

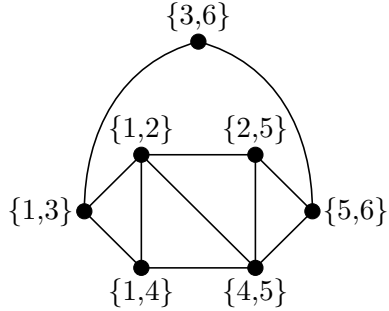
\node[align=center, below] at  (1.5,-0.7){Fig. 2: The graph  $\mathcal{G}$, Case 36};
 \end{tikzpicture}
 \end{center}

\noindent{\bf Sets no. 37-42. } In each of these cases there are $7$ minimal supports, each of them, except possibly one, of size $3$.
 By Lemma \ref{lem:size 3 supp union}, each zero support of $M$ is a union of at most two minimal supports.
Thus Lemma \ref{lem:XXT tf} implies that $\cpr(A)\le \max(7,\tf(\mathcal{G}_{\mathcal{V}}(M)))$.
Combined with (\ref{eq:G(M)}) we get
that
\[\cpr (A)\le \max(7, \tf(\mathcal{G}(M))).\]
In Figs. 3--8 the graph $\mathcal{G}(M)$ is shown for each of these cases. In all of them $\tf(\mathcal{G}(M))\le 9$.
(In each case, $\tf(\mathcal{G})(M)$ turns out to be $|E(\mathcal{G}(M))|-q$, where $q=1$ or $2$ is
the maximal number of edge-disjoint triangles in the graph.)

\vskip 4mm

\begin{center}
 \begin{tikzpicture}


\draw[semithick] (0,1.25)--(2.25,0);
\draw[semithick] (0,1.25)--(4.5,1.25);
\draw[semithick] (4.5,1.25)--(2.25,0);
\draw[semithick] (0,1.25)--(1.5,2.75);
\draw[semithick] (1.5,2.75)--(3,2.75);
\draw[semithick] (3,2.75)--(4.5,1.25);
\draw[semithick] (1.5,2.75)--(1.5,1.25);
\draw[semithick] (3,2.75)--(1.5,1.25);
\draw[semithick] (3,2.75)--(3,1.25);

\draw[fill](2.25,0) circle[radius=0.1];
\draw[fill](0,1.25) circle[radius=0.1];
\draw[fill](4.5,1.2) circle[radius=0.1];
\draw[fill](3,1.25) circle[radius=0.1];
\draw[fill](3,2.75) circle[radius=0.1];
\draw[fill](1.5,2.75) circle[radius=0.1];
\draw[fill](1.5,1.25) circle[radius=0.1];

\node[below] at (2.25,0) {\{1,4,6\}};
\node[left] at (0,1.25) {\{4,5,6\}};
\node[right] at (4.5,1.25){\{1,3,4\}};
\node[above left] at (1.5,2.75) {\{2,5,6\}};
\node[above right] at (3,2.75) {\{1,2\}};
\node[below ] at (1.5,1.25) {\{3,5,6\}};
\node[below ] at (3,1.25){\{1,3,5\}};

\node[below] at  (2.2,-1) {Fig. 3: Case 37, $\tf(\mathcal{G}(M))=9$};


\draw[semithick] (8,0)--(12,0);
\draw[semithick] (8,0)--(10,2.65);
\draw[semithick] (12,0)--(10,2.65);
\draw[semithick] (10,0)--(9,1.35);
\draw[semithick] (10,0)--(11,1.35);
\draw[semithick] (8.4,1.9)--(10,2.65);
\draw[semithick] (8.4,1.9)--(8,0);

\draw[fill](8,0) circle[radius=0.1];
\draw[fill](12,0) circle[radius=0.1];
\draw[fill](10,2.65) circle[radius=0.1];
\draw[fill](8.4,1.9) circle[radius=0.1];
\draw[fill](11,1.35) circle[radius=0.1];
\draw[fill](10,0) circle[radius=0.1];
\draw[fill](9,1.35) circle[radius=0.1];

\node[left] at (8.4,1.9) {\{3,5,6\}};
\node[below left] at (8,0) {\{2,5,6\}};
\node[below] at (10,0){\{1,2\}};
\node[above] at (10,2.65) {\{3,4,6\}};
\node[right ] at (11,1.35) {\{1,3,4\}};
\node[below right] at (12,0) {\{1,3,5\}};
\node[right] at (9,1.35){\{2,4,6\}};

\node[align=center, below] at  (9.75,-1){Fig. 4: Case 38, $\tf(\mathcal{G}(M))=8$};
 \end{tikzpicture}
 \end{center}

\vskip 4mm

\begin{center}
 \begin{tikzpicture}


\draw[semithick] (1.5,0)--(3,0);
\draw[semithick] (0.75,1.25)--(3.75,1.25);
\draw[semithick] (0.75,1.25)--(1.5,0);
\draw[semithick] (3.75,1.25)--(3,0);
\draw[semithick] (0.75,2.5)--(2.25,2.5);
\draw[semithick] (2.25,2.5)--(3.75,1.25);
\draw[semithick] (2.25,1.25)--(2.25,2.5);
\draw[semithick] (0.75,1.25)--(0.75,2.5);

\draw[fill](1.5,0) circle[radius=0.1];
\draw[fill](3,0) circle[radius=0.1];
\draw[fill](0.75,1.25) circle[radius=0.1];
\draw[fill](3.75,1.25) circle[radius=0.1];
\draw[fill](2.25,1.25) circle[radius=0.1];
\draw[fill](2.25,2.5) circle[radius=0.1];
\draw[fill](0.75,2.5) circle[radius=0.1];

\node[above left] at (0.75,2.5) {\{1,4,6\}};
\node[below left] at (1.5,0) {\{3,5,6\}};
\node[left] at (0.75,1.25){\{1,3,6\}};
\node[above] at (2.25,2.5) {\{1,2,4\}};
\node[below] at (2.25,1.25) {\{1,2,3\}};
\node[below right] at (3,0) {\{2,5,6\}};
\node[above right] at (3.75,1.25){\{1,2,5\}};

\node[below] at  (2.2,-1) {Fig. 5: Case 39, $\tf(\mathcal{G}(M))=8$};


\draw[semithick] (8,0)--(11,0);
\draw[semithick] (8,1.5)--(11,1.5);
\draw[semithick] (9.5,0)--(9.5,1.5);
\draw[semithick] (8,0)--(8,1.5);
\draw[semithick] (11,0)--(11,1.5);
\draw[semithick] (11,1.5)--(12,0.75);
\draw[semithick] (11,0)--(12,0.75);

\draw[fill](8,0) circle[radius=0.1];
\draw[fill](8,1.5) circle[radius=0.1];
\draw[fill](11,0) circle[radius=0.1];
\draw[fill](11,1.5) circle[radius=0.1];
\draw[fill](9.5,0) circle[radius=0.1];
\draw[fill](9.5,1.5) circle[radius=0.1];
\draw[fill](12,0.75) circle[radius=0.1];

\node[above left] at (8,1.5) {\{4,5,6\}};
\node[below left] at (8,0) {\{3,5,6\}};
\node[below] at (9.5,0){\{1,3,6\}};
\node[above] at (9.5,1.5) {\{1,4,6\}};
\node[right ] at (12,0.75) {\{1,2,5\}};
\node[below right] at (11,0) {\{1,2,3\}};
\node[above right] at (11,1.5){\{1,2,4\}};

\node[below] at  (9.75,-1) {Fig. 6: Case 40, $\tf(\mathcal{G}(M))=8$};

 \end{tikzpicture}
 \end{center}

\vskip 4mm

\begin{center}
 \begin{tikzpicture}


\draw[semithick] (0,1.25)--(0.75,2);
\draw[semithick] (0,1.25)--(0.75,0.5);
\draw[semithick] (0.75,0.5)--(0.75,2);
\draw[semithick] (0.75,0.5)--(3.75,0.5);
\draw[semithick] (0.75,2)--(3.75,2);
\draw[semithick] (3.75,0.5)--(3.75,2);
\draw[semithick] (3.75,2)--(4.75,1.25);
\draw[semithick] (3.75,0.5)--(4.75,1.25);

\draw[fill](0.75,0.5) circle[radius=0.1];
\draw[fill](0.75,2) circle[radius=0.1];
\draw[fill](3.75,0.5) circle[radius=0.1];
\draw[fill](3.75,2) circle[radius=0.1];
\draw[fill](0,1.25) circle[radius=0.1];
\draw[fill](2.25,2) circle[radius=0.1];
\draw[fill](4.75,1.25) circle[radius=0.1];

\node[above left] at (0.75,2) {\{3,4,6\}};
\node[below left] at (0.75,0.5) {\{1,3,6\}};
\node[above] at (2.25,2){\{2,4,6\}};
\node[left] at (0,1.25) {\{3,5,6\}};
\node[right ] at (4.75,1.25) {\{1,2,5\}};
\node[below right] at (3,0.5) {\{1,2,3\}};
\node[above right] at (3.75,2){\{1,2,4\}};

\node[below] at  (2.25,-0.6) {Fig. 7: Case 41, $\tf(\mathcal{G}(M))=7$};


\draw[semithick] (8,0.5)--(8,2);
\draw[semithick] (8,0.5)--(11,0.5);
\draw[semithick] (8,2)--(11,2);
\draw[semithick] (11,0.5)--(11,2);
\draw[semithick] (11,2)--(12,1.25);
\draw[semithick] (11,0.5)--(12,1.25);

\draw[fill](8,0.5) circle[radius=0.1];
\draw[fill](8,2) circle[radius=0.1];
\draw[fill](11,0.5) circle[radius=0.1];
\draw[fill](11,2) circle[radius=0.1];
\draw[fill](12,1.25) circle[radius=0.1];
\draw[fill](9.5,2) circle[radius=0.1];
\draw[fill](9.5,0.5) circle[radius=0.1];

\node[above left] at (8,2) {\{4,5,6\}};
\node[below left] at (8,0.5) {\{3,5,6\}};
\node[above] at (9.5,2){\{2,4,6\}};
\node[below] at (9.5,0.5){\{1,3,6\}} ;
\node[right ] at (12,1.25) {\{1,2,5\}};
\node[below right] at (11,0.5) {\{1,2,3\}};
\node[above right] at (11,2){\{1,2,4\}};

\node[below] at  (9.75,-0.6) {Fig. 8: Case 42, $\tf(\mathcal{G}(M))=7$};

 \end{tikzpicture}
 \end{center}

\noindent{\bf Set no. 43. }
In this case, the matrix $M$ has $8$ minimal supports of size $3$, and by Lemma \ref{lem:size 3 supp union}, each zero support is the union
of at most two minimal supports.  The graph $\mathcal{G}(M)$ for this case is shown in Fig. 9.

\begin{center}
 \begin{tikzpicture}

\draw[semithick] (0,1.5)--(1.5,0);
\draw[semithick] (1.5,0)--(4.5,0);
\draw[semithick] (1.5,0)--(1.5,1.5);
\draw[semithick] (4.5,0)--(6,1.5);
\draw[semithick] (4.5,0)--(4.5,1.5);
\draw[semithick] (0,1.5)--(1.5,1.5);
\draw[semithick] (4.5,1.5)--(6,1.5);
\draw[semithick] (1.5,1.5)--(3,3);
\draw[semithick] (0,1.5)--(3,4.5);
\draw[semithick] (3,3)--(3,4.5);
\draw[semithick] (4.5,1.5)--(3,3);
\draw[semithick] (6,1.5)--(3,4.5);

\draw[fill](1.5,0) circle[radius=0.1];
\draw[fill](0,1.5) circle[radius=0.1];
\draw[fill](4.5,1.5) circle[radius=0.1];
\draw[fill](4.5,0) circle[radius=0.1];
\draw[fill](1.5,1.5) circle[radius=0.1];
\draw[fill](3,3) circle[radius=0.1];
\draw[fill](3,4.5) circle[radius=0.1];
\draw[fill](6,1.5) circle[radius=0.1];

\node[left] at (0,1.5) {\{3,5,6\}};
\node[below left] at (1.5,0) {\{2,5,6\}};
\node[left ] at (4.5,1.5) {\{1,2,4\}};
\node[right] at (1.5,1.5) {\{4,5,6\}};
\node[below right] at (4.5,0) {\{1,2,5\}};
\node[below] at (3,2.6){\{1,4,6\}};
\node[right] at (6,1.5) {\{1,2,3\}};
\node[above] at (3,4.5) {\{1,3,6\}};

\node[align=center, below] at  (3,-0.7){Fig. 9: $\mathcal{G}(M)$, Case 43};
 \end{tikzpicture}
 \end{center}

For every $i,j\in \{1,2,3,5,6\}$, $i\ne j$, $\{i,j\}$ is a subset of
one of the minimal zeros. Thus by Proposition \ref{pro:Nirr},
the matrix $M(4)$ is $\tilde{\nn}$-irreducible. Thus if $M(4)$ had a zero support of size $4$, then $M(4)$ would be  positive semidefinite by Proposition
\ref{pro:maxsizesupp}, and then, since $M$ itself is  $\tilde{\nn}$-irreducible, $M$ would also be positive semidefinite by Proposiition \ref{pro:n-1psd of irr}, contrary
to the assumption that $M$ is exceptional. Thus there are no zero supports of $M$ of size $4$ contained in $\{1,2,3,5,6\}$.
By the same argument for $M(3)$, there are no zero supports of $M$ of size $4$ contained in  $\{1,2,4,5,6\}$. Thus $\mathcal{G}_{\mathcal{V}}(M)$
is actually a subgraph of the smaller graph shown in Fig. 10, which is a forest. By (\ref{eq:cprAcprGVM}), $\cpr(A)\le \cpr(\mathcal{G}_{\mathcal{V}}(M))\le 8$.

\begin{center}
 \begin{tikzpicture}

\draw[semithick] (0,1.5)--(1.5,1.5);
\draw[semithick] (4.5,1.5)--(6,1.5);
\draw[semithick] (3,3)--(3,4.5);

\draw[fill](1.5,0) circle[radius=0.1];
\draw[fill](0,1.5) circle[radius=0.1];
\draw[fill](4.5,1.5) circle[radius=0.1];
\draw[fill](4.5,0) circle[radius=0.1];
\draw[fill](1.5,1.5) circle[radius=0.1];
\draw[fill](3,3) circle[radius=0.1];
\draw[fill](3,4.5) circle[radius=0.1];
\draw[fill](6,1.5) circle[radius=0.1];

\node[left] at (0,1.5) {\{3,5,6\}};
\node[below left] at (1.5,0) {\{2,5,6\}};
\node[left ] at (4.5,1.5) {\{1,2,4\}};
\node[right] at (1.5,1.5) {\{4,5,6\}};
\node[below right] at (4.5,0) {\{1,2,5\}};
\node[below] at (3,2.6){\{1,4,6\}};
\node[right] at (6,1.5) {\{1,2,3\}};
\node[above] at (3,4.5) {\{1,3,6\}};

\node[align=center, below] at  (3,-0.7){Fig. 10: A supergraph of $\mathcal{G}_{\mathcal{V}}(M)$, Case 43};
 \end{tikzpicture}
 \end{center}

\noindent{\bf Set no. 44. }
In this case, the matrix $M$ has $8$ minimal supports, all of size $3$. The graph $\mathcal{G}(M)$  is the bipartite graph shown in Fig. 11 (the cube graph).
Suppose there is a path of length two  in the inner $4$-cycle, such that each of its edges represents a zero support of size $4$ of $M$. Then
for every $i,j\in \{1,2,3,4,5\}$, $i\ne j$, $\{i,j\}$ is a subset of
a zero support of $M$, and therefore of $M(6)$. Thus the principal submatrix $M(6)$ is a $5\times 5$ $\tilde{\nn}$-irreducible matrix with a zero support of size $4$. By
Proposition \ref{pro:maxsizesupp}, $M(6)$ is then positive semidefinite. But
then $M$ itself is positive semidefinite by Proposition \ref{pro:n-1psd of irr}, contrary to the fact that $M$ is exceptional.
Thus at most two parallel edges of the inner $4$-cycle in Fig. 11 represent zeros of  size $4$ of $M$. By the same argument  for $M(1)$ at most
two parallel edges of the outer $4$-cycle represent zeros of  size $4$ of $M$. That is,  at most $8$ of the $12$ edges of the  graph  $\mathcal{G}(M)$
shown in Fig. 11 are edges of $\mathcal{G}_{\mathcal{V}}(M)$. Hence $\cpr(A)\le \cpr(\mathcal{G}_{\mathcal{V}}(M))\le 8$.\hfill$\square$

\begin{center}
 \begin{tikzpicture}

\draw[semithick] (0,0)--(5,0);
\draw[semithick] (0,3.5)--(5,3.5);
\draw[semithick] (1,1)--(4,1);
\draw[semithick] (1,2.5)--(4,2.5);
\draw[semithick] (1,1)--(1,2.5);
\draw[semithick] (4,1)--(4,2.5);
\draw[semithick] (0,0)--(0,3.5);
\draw[semithick] (5,0)--(5,3.5);
\draw[semithick] (0,0)--(1,1);
\draw[semithick] (0,3.5)--(1,2.5);
\draw[semithick] (5,0)--(4,1);
\draw[semithick] (5,3.5)--(4,2.5);

\draw[fill](0,0) circle[radius=0.1];
\draw[fill](5,0) circle[radius=0.1];
\draw[fill](0,3.5) circle[radius=0.1];
\draw[fill](5,3.5) circle[radius=0.1];
\draw[fill](1,1) circle[radius=0.1];
\draw[fill](4,1) circle[radius=0.1];
\draw[fill](1,2.5) circle[radius=0.1];
\draw[fill](4,2.5) circle[radius=0.1];

\node[below left] at (0,0) {\{3,5,6\}};
\node[below right] at (5,0)  {\{4,5,6\}};
\node[above left] at (0,3.5)  {\{2,3,6\}};
\node[above right] at (5,3.5) {\{2,4,6\}};
\node[above right] at (1,1)  {\{1,3,5\}};
\node[above left] at (4,1)  {\{1,4,5\}};
\node[below right] at (1,2.5)  {\{1,2,3\}};
\node[below left] at (4,2.5) {\{1,2,4\}};

\node[align=center, below] at  (2.5,-0.7){Fig. 11:  $\mathcal{G}(M)$, Case 44};
 \end{tikzpicture}
 \end{center}

Note that by Proposition \ref{pro:tf} a completely positive matrix $A$ whose graph is the complete bipartite graph
$K_{3,3}$ has $\cpr(A)=|E(K_{3,3})|=9$. Since $p_6$ is attained at a nonsingular matrix on the boundary, this together with
Theorem \ref{thm:main} implies the following.

\begin{corollary}\label{cor:p6 attained}
The maximum cp-rank $p_6$ is attained at a nonsingular matrix  $A\in \cp_6$ which has a zero entry.
\end{corollary}

\bibliographystyle{plain}

\end{document}